\documentclass[final]{siamart190516}

\usepackage[text={5.125in,8.25in},centering]{geometry}

\ifpdf
\hypersetup{
  pdftitle={Power of two choices for Boolean functions},
  pdfauthor={N. Fraiman, L. Lichev, and D. Mitsche}
}
\fi
\hypersetup{colorlinks = true, linkcolor = blue, citecolor = blue, urlcolor = blue}

\usepackage{amsfonts}
\usepackage{amssymb,bm}
\usepackage{dsfont}

\renewcommand{\Pr}[1]	{\mathbf{P}\!\bm{\left(} #1 \bm{\right)}}

\newcommand{\N}     	{\mathbb{N}}

\newcommand{\eps}		{\varepsilon}
\newcommand{\cupdot}	{\mathbin{\mathaccent\cdot\cup}}

\newsiamremark{remark}{Remark}

\title{On the power of choice for Boolean functions\thanks{
\funding{Dieter Mitsche has been supported by grant GrHyDy ANR-20-CE40-0002 and by IDEXLYON of Universit\'{e} de Lyon (Programme Investissements d'Avenir ANR16-IDEX-0005).}}}

\author{Nicolas Fraiman\thanks{University of North Carolina, Chapel Hill, NC, United States
  (\email{fraiman@email.unc.edu}).}
\and Lyuben Lichev\thanks{Institut Camille Jordan, Univ.~Lyon 1, Lyon, France. Univ.~Jean Monnet, Saint-Etienne, France
  (\email{lyuben.lichev@univ-st-etienne.fr}, \email{dmitsche@gmail.com}).}
\and Dieter Mitsche\footnotemark[3]}

\headers{On the power of choice for Boolean functions}{N. Fraiman, L. Lichev, and D. Mitsche}

\begin{document}

\maketitle

\begin{abstract}
    In this paper we consider a variant of the well-known Achlioptas process for graphs adapted to monotone Boolean functions. Fix a number of choices $r\in \mathbb N$ and a sequence of increasing functions $(f_n)_{n\ge 1}$ such that, for every $n\ge 1$, $f_n:\{0,1\}^n\mapsto \{0,1\}$. Given $n$ bits which are all initially equal to 0, at each step $r$ 0-bits are sampled uniformly at random and are proposed to an agent. Then, the agent selects one of the proposed bits and turns it from 0 to 1 with the goal to reach the preimage of 1 as quickly as possible. We nearly characterize the conditions under which an acceleration by a factor of $r(1+o(1))$ is possible, and underline the wide applicability of our results by giving examples from the fields of Boolean functions and graph theory.
\end{abstract}



%=======================================
\section{Introduction}

The ``power of two choices'' was introduced by Azar, Broder, Karlin and Upfal~\cite{ABKU} in the context of load balancing. They showed that, when randomly allocating $n$ balls into $n$ bins, a dramatic decrease in the maximum load is achieved by sequentially selecting the less full bin among two random options. Many variations on this basic model have been analyzed. Berenbrink, Czumaj, Steger and V\"ocking~\cite{BCSV} studied the case when a much larger number of balls is placed. Kenthapadi and Panigrahy~\cite{KP} restricted the options by placing balls in an endpoint of a random edge from a graph. More recently, Redlich~\cite{Red} studied the case where you want to ``unbalance'' and select the fullest bin.

A classical and well-studied setting is the \emph{Erd\H{o}s-R\'enyi graph process} where the edges of the complete graph $K_n$ arrive one by one according to a uniform random permutation. The power of choice in this context was introduced by Achlioptas: he was interested in the question of delaying certain monotone graph properties with respect to the original process if at each step, $r \ge 2$ edges instead of one are proposed and an agent may choose the one they need more for their purposes (we call this variation the \emph{$r$--choice process})\footnote{At the last $r-1$ steps, all 0-bits are proposed.}. In two related papers Bohman and Frieze~\cite{BF} and Spencer and Wormald~\cite{SW} studied the problem of delaying the appearance of a giant component by the $r$--choice process. Krivelevich, Loh and Sudakov~\cite{KLS} studied rules to avoid small subgraphs. Achlioptas, D'Souza and Spencer~\cite{ADS} claimed that certain rules could make the giant transition discontinuous but Riordan and Wernke~\cite{RW} proved that was not the case. A more restrictive version where the agent's decisions cannot depend on the previous history and only one vertex from the random edges is revealed was studied by Beveridge, Bohman, Frieze, and Pikhurko~\cite{BBFP}. A similar restrictive model is the so called semi-random graph process, where one vertex is chosen randomly and the agent can choose the second vertex arbitrarily, see the paper of Ben-Eliezer, Hefetz, Kronenberg, Parczyk, Shikhelman and Stojakovi\'{c}~\cite{BenEliezer}. When the goal is to expedite rather than delay certain properties, Krivelevich and Spöhel~\cite{KS} proved general upper and lower bounds on the threshold to create a copy of some fixed graph $H$ in the $r$--choice process. Recently, the question of acceleration of the appearance of a Hamilton Cycle or a Perfect Matching was treated by Krivelevich, Lubetzky and Sudakov~\cite{Krive} who proved that there exist strategies that accelerate both properties by a factor of $r+o(1)$.
Furthermore, outside of the graph setup, Sinclair and Vilenchik~\cite{SV} turned particular attention to delaying the satisfiability of the random 2-SAT formula, and Perkins~\cite{WP} considered a $k$-SAT version of the problem. 

In their seminal work, Erd\H{o}s and R\'enyi \cite{ER} showed that many interesting graph properties exhibit sharp thresholds, that is, the probability that a random graph with $n$ vertices and $m$ edges has the property increases from values very close to 0 to values close to 1 in a very small interval around a certain critical value of the number of edges $m$ (often called a critical window). Later, Bollob\'as and Thomason \cite{BT} proved the existence of threshold functions for all monotone graph properties. A more careful analysis of the size of the critical window was performed by Friedgut and Kalai \cite{FK}.

Their arguments generalize in a straightforward way to thresholds of monotone Boolean functions. More precisely, for any $n \ge 1$, consider the hypercube $\{0,1\}^n$ with the probability measure $\mu_p(x_1,\dots,x_n) = p^k (1-p)^{n-k}$ where $k=x_1+\cdots+x_n$. Let $(A_n)_{n \ge 1}$ be a sequence of monotone sets such that, for every $n\ge 1$, $A_n\subseteq \{0,1\}^n$ and $A_n$ is invariant under a transitive permutation group of $\{1,2,\dots,n\}$. If $\mu_p(A_n) > \epsilon$, then Bollob\'as and Thomason \cite{BT} showed that there is $c(\eps) > 0$ such that $\mu_q(A_n) > 1-\epsilon$ for $q = c(\epsilon)p$. This result was improved by Friedgut and Kalai \cite{FK} to $\mu_q(A_n) > 1-\epsilon$ for $q = p + c \log(1/2\epsilon)/\log n$, where $c$ is an absolute constant. We say that a function is a \emph{sharp threshold function} for the sequence of monotone subsets $(A_n)_{n \ge 1}$ if, for every $\eps > 0$, the probability $p_n$ such that $\mu_{p_n}(A_n)=\eps$ and the probability $q_n$ such that $\mu_{q_n}(A_n)=1-\eps$ satisfy $p_n = (1+o(1))q_n$. Then, the threshold function is given only up to a $(1+o(1))$ factor by both $(p_n)_{n\ge 1}$ and $(q_n)_{n\ge 1}$ for any fixed $\eps > 0$. Equivalently, the hitting time of the event $A_n$ by the process that turns from 0 to 1 the $n$ given bits one by one in an order, chosen uniformly at random, is of order $(1+o(1))p_n$ asymptotically almost surely. Sharp thresholds appear in various systems in combinatorics, computer science and statistical physics (where they are more widely known as phase transitions).

Motivated by all these questions, we embark in the study of the power of choice for Boolean functions. Our goal is to characterize Boolean functions whose thresholds can be maximally accelerated. More precisely, we study the $r$--choice process for Boolean functions where at each step an agent is presented with $r$ zero coordinates and selects one to flip (here and below, $r\ge 1$ is a fixed positive integer). Our objective is to understand which monotone Boolean functions can be accelerated by a factor of $r$ by the $r$--choice process (as we shall see in a bit, the factor $r$ is optimal). For that purpose, we compare the hitting probabilities for the function to reach the value $1$ under two increasing random walks on the hypercube.

The paper is organized as follows. In Section \ref{sec:res} we introduce the model of interest, state the assumptions and present the main results of the paper, which are then proved in Section \ref{sec:proofs}. Section \ref{sec:app} contains concrete applications of our results.

%=======================================
\section{Statements of results}\label{sec:res}

We use the following standard asymptotic notation: for two sequences of functions $(a_n)_{n \ge 1}$ and $(b_n)_{n \ge 1}$ we say that $a_n=O(b_n)$ if there exists $C > 0$ and $n_0 \in \mathbb{N}$ such that, for all $n\ge n_0$, $|a_n|\le C|b_n|$; $a_n=\Omega(b_n)$ if $b_n=O(A_n)$; $a_n=\Theta(b_n)$ if $a_n=O(b_n)$ and $b_n=\Omega(a_n)$; $a_n=o(b_n)$ or equivalently $a_n \ll b_n$ if $\lim_{n \to \infty} \frac{|a_n|}{|b_n|}=0$; and $a_n=\omega(b_n)$ if $b_n=o(a_n)$. In case the limit is taken with respect to a different variable $k$, we use the notation $o_k(b_n)$, $\Omega_k(b_n)$, etc. to point this out. We also say that a sequence of events $(E_n)_{n \ge 1}$ holds a.a.s.\ (or asymptotically almost surely), if $\lim_{n \to \infty} \Pr{E_n}=1$.

Fix any $n\in \mathbb N$. A \emph{Boolean} function $f$ maps elements from the hypercube $\{0,1\}^n$ to $\{0,1\}$. We denote the vectors in $\{0,1\}^n$ by lower case letters in bold such as $\bm{u}, \bm{v}, \bm{w},$ etc. For a vector $\bm{x}$, we denote by $|\bm{x}|$ the number of coordinates of $\bm{x}$. We denote by $\bm{0}$ the all zeroes vector and by $\bm{1}$ the all ones vector. 

We will see the hypercube as a partially ordered set equipped with the order relation $\le$ defined by $\bm{x}\le \bm{y}$ if $\bm{x}_i\le \bm{y}_i$ for every $i \in [n]$. At the same time, construct an oriented edge between every pair of vectors $\bm{x},\bm{y}\in \{0,1\}^n$ such that $\bm{x}\le \bm{y}$, and $\bm{x}$ and $\bm{y}$ differ in exactly one coordinate - this allows us in turn to see the hypercube as a directed graph.

A Boolean function is \emph{monotone} if $\bm{x}\le \bm{y}$ implies $f(\bm{x})\le f(\bm{y})$.

\begin{definition}
    A variable $i$ is \emph{relevant} for $f$ if there exist inputs $\bm{x},\bm{y}\in\{0,1\}^n$ which differ only in coordinate $i$ and $f(\bm{x})\neq f(\bm{y})$; in this case we also say that $f$ \emph{depends on} the $i$-th variable. The \emph{relevant set of $f$}, denoted by $R(f)$, is the set of variables relevant for $f$.
\end{definition}

\begin{definition}
The \emph{relevant contraction} of a Boolean function $f$, denoted by $\tilde f$, is the function obtained by restricting $f$ to its relevant set. In other words, if $f:\{0,1\}^n \to \{0,1\}$ and $R(f)=\{i_1,\dots,i_m\}$, then $\tilde{f}:\{0,1\}^m\to\{0,1\}$ is defined as $\tilde{f}(\bm{x}) = f(\bm{y})$ where $\bm{y}_{i_j} = \bm{x}_j$ for $j=1,\dots,m$, and for every $i\in [n]\setminus \{i_1,\dots,i_m\}, \bm{y}_i$ is an arbitrary bit. 
\end{definition}

We will be interested in two random walks on the (directed) hypercube $\{0,1\}^n$. The simple random walk $(\bm{X}_t)_{t=0}^n$ starts at $\bm{X}_0 = \bm{0}$ and evolves by choosing a directed edge uniformly at random and moves in its direction at each step. In the $r$--choice walk $(\bm{Y}_t)_{t=0}^{n}$ starting from $\bm{Y}_0 = \bm{0}$, an agent is presented with $r$ zero bits chosen uniformly at random, selects one of them and moves in its direction (in the end when there are fewer than $r$ possible edges, we assume that all zero bits are proposed). Formally, for every integer $t\in [0,n-r]$, let $Z_t$ be the set of zero coordinates in $\bm{Y}_t$, and let $C_t$ be the random subset of $Z_t$ of size $r$, presented to the agent at step $t$. Then, the agent selects $c_t\in C_t$ according to some policy and updates the set of zero coordinates $Z_{t+1} = Z_t \setminus \{c_t\}$. Given a monotone Boolean function $f$, we will study the hitting times of the preimage $f^{-1}(1)\subset \{0,1\}^n$ by the two processes $(\bm{X}_t)_{t=0}^n$ and $(\bm{Y}_t)_{t=0}^n$ (at this moment we say that the function $f$ is \emph{activated}).

\begin{definition}\label{def:thresholds}
    The \emph{solo} and the \emph{$r$--choice thresholds} are given by
    \begin{align*}
        T_1(f) &= \min\big\{t: \Pr{f(\bm{X}_t)=1} \ge 1/2\big\},\\
        T_r(f) &= \min\big\{t: \Pr{f(\bm{Y}_t)=1} \ge 1/2\big\}.
    \end{align*}
\end{definition}

In this paper, when we are talking about a sequence of Boolean functions $(f_n)_{n\ge 1}$, we will always assume that $f_n:\{0,1\}^n\to \{0,1\}$ is monotone unless explicitly mentioned otherwise. The main question we consider is if one may asymptotically accelerate by a factor $r$ the threshold values for the $r$--choice process (unless explicitly stated otherwise, all asymptotics refer to the regime $n\to +\infty$).

\begin{definition}
    A sequence of functions $(f_n)_{n\ge 1}$ is \emph{fast} if
    \begin{equation*}
        T_r(f_n) = (1+o(1)) \frac{T_1(f_n)}{r}.
    \end{equation*}    
    A sequence is \emph{slow} if it is not fast.
\end{definition}

Notice that the constant $r$ is best possible: indeed, define the $r$--complete process to be the process, in which one changes all $r$ uniformly chosen remaining zeros to $1$ at the same time. This process performs only a $1/r$-fraction of the time steps of the single choice process, and is at least as fast as the $r$--choice process.

We need one more definition that allows us to formalize the concept that relevant sets of variables might change over the process. For every $n\ge 1$, the sequence of functions $(f^s_n)_{s\ge 0}$ is defined conditionally on the sequence of updated bits $(b_s)_{s\ge 1}$ as follows. Order the first $s$ bits in increasing order $b_{i_1} < \dots < b_{i_s}$. For every integer $s\in [0,n]$ and a vector $\bm{v}\in \{0,1\}^{n-s}$, define
\begin{equation*}
\bm{v}_s^\uparrow = (\bm{v}_1,\dots, \bm{v}_{b_{i_1}-1}, 1, \bm{v}_{b_{i_1}},\dots , \bm{v}_{b_{i_2}-2}, 1, \bm{v}_{b_{i_2}-1},\dots , \bm{v}_{b_{i_s}-s}, 1, \bm{v}_{b_{i_s}-s+1},\dots , \bm{v}_{n-s}).
\end{equation*}
Define $f^s_n: \bm{v}\in \{0,1\}^{n-s}\to f_n(\bm{v}_s^\uparrow)\in \{0,1\}$. In particular, $f^0_n = f_n$. Observe that $(|R(f^s_n)|)_{s = 0}^n$ is a non-increasing sequence since for any fixed integer $s\in [0,n-1]$, if a position $i$ is not in the set $R(f^s_n)$, then it remains outside the set $R(f^{s+1}_n)$ as well.

We now present our main results. Throughout we fix an integer $r\ge 2$. We first state two sufficient conditions for a sequence $(f_n)_{n\ge 1}$ to be slow.

\begin{theorem}\label{thm:slow}
If there is $\eps > 0$ such that $T_1(f_n)\ge \eps n$ for every $n\ge 1$. Then, there exists a constant $C = C(r) > 0$ such that, for every $n\ge 1$, $T_r(f_n)\ge C n + T_1(f_n)/r$.
\end{theorem}

\begin{corollary}\label{cor:slow}
If $|R(f_n)| = \omega(1)$, and there is $\delta > 0$ such that $T_1(\tilde{f}_n)\ge \delta |R(f_n)|$ for every $n\ge 1$.
Then, there exists a constant $C = C(r) > 0$ such that, for every $n\ge 1$, $T_r(f_n)\ge C n + T_1(f_n)/r$.
\end{corollary}

Now, we state two sufficient conditions for a sequence $(f_n)_{n\ge 1}$ to be fast.

\begin{theorem}\label{thm:fast}
If $1\ll T_1(\tilde{f}_n)\ll |R(f_n)|\ll n$, then $T_r(f_n) = (1+o(1)) T_1(f_n)/r$.
\end{theorem}

\begin{corollary}\label{cor:fast}
Suppose that a.a.s.\ for every $\eps > 0$ there is $s = s(n)$ such that:
\begin{enumerate}
    \item $s \le \eps T_1(f_n)$,
    \item $|R(f^s_n)| \le \eps n$, 
    \item $\dfrac{1}{\eps}\le T_1(\tilde{f}^s_n) \le \eps |R(f^s_n)|$.
\end{enumerate}
Then, $T_r(f_n) = (1+o(1)) T_1(f_n)/r$.
\end{corollary}
In the following sections, we often omit upper and lower integer parts when rounding does not matter in the corresponding computation.

%=======================================
\section{Proofs of the main results}\label{sec:proofs}

We split this section into two parts with the results characterizing slow and fast sequences respectively.
%==================
\subsection{Slow sequences}

We present the proofs of Theorem \ref{thm:slow} and Corollary \ref{cor:slow}.

\begin{lemma}\label{lem:collision}
Fix any $\eps \in (0,1)$ and $c\in (0, (1-\eps)^r)$. Then, in $\eps n$ steps of the $r$--choice process there are at least $cn$ elements that have been proposed at least twice a.a.s.
\end{lemma}
\begin{proof}
The probability that a given element $i\in [n]$ has never been proposed by the $r$--choice process up to step $\eps n$ is given by
\begin{align*}
    & \prod_{0\le i\le \eps n-1} \prod_{0\le j\le r-1}\left(1 - \dfrac{1}{n-i-j}\right)\\ =\hspace{0.3em}
    & (1+o(1))\prod_{0\le i\le \eps n-1} \left(1 - \dfrac{r}{n-i}\right)\\ =\hspace{0.3em}
    & (1+o(1))\exp\left(-r\log\left(\dfrac{n}{n-\eps n}\right)\right)\\ =\hspace{0.3em} 
    & (1+o(1))\left(1 - \eps\right)^r.
\end{align*}
Also, the probability that two different elements have both not been proposed after $\eps n$ steps
is 
\begin{align*}
    & \prod_{0\le i\le \eps n-1} \prod_{0\le j\le r-1}\left(1 - \dfrac{2}{n-i-j}\right)\\ =
    & (1+o(1))\prod_{0\le i\le \eps n-1} \left(1 - \dfrac{2r}{n-i}\right)\\ =
    & (1+o(1))\exp\left(-2r\log\left(\dfrac{n}{n-\eps n}\right)\right)\\ = 
    & (1+o(1))\left(1 - \eps\right)^{2r}\\ =
    & ((1+o(1))\left(1 - \eps\right)^r)^2.
\end{align*}

We conclude by a direct application of the second moment method that the number of vertices not yet proposed during any of the first $\eps n$ steps, is a.a.s.\ at least $cn$, which proves the proposition.
\end{proof}

\begin{proof}[Proof of Theorem~\ref{thm:slow}]
We argue by contradiction. In this case there is an increasing sequence $(n_k)_{k\ge 1}$ such that $rT_r(f_{n_k}) = (1+o_k(1))T_1(f_{n_k})$. Since $T_r(f_n) \ge T_1(f_n)/r \ge \eps n/r$ for every $n\ge 1$, by Lemma~\ref{lem:collision} there is $c > 0$ such that a.a.s.\ at least $c n$ elements have been proposed at least twice by the $r$--choice process until step $T_r(f_n)$. Hence, for every $n\ge 1$, the number of all elements that have been proposed at least once up to time $T_r(f_n)$ in the $r$--choice process is a.a.s.\ at most $rT_r(f_n) - cn$. Thus, for all $k\ge 1$, the number of bits proposed by the $r$--choice process up to step $T_r(f_{n_k})$ (out of all $n_k$ bits) is at most $T_1(f_{n_k}) - cn_k + o_k(n_k)$ a.a.s.\ and, conditionally on their number, these are chosen uniformly at random. Therefore, the probability that $f_{n_k}$ is activated by the above set of elements is less than $1/2$ for every large enough $k$, which is in contradiction with our assumption.
\end{proof}

\begin{proof}[Proof of Corollary~\ref{cor:slow}]
Fix $n\ge 1$. If $|R(f_n)|\ge n/8$, then $T_1(f_n) \ge T_1(\tilde{f}_n) \ge \delta |R(f_n)| \ge \delta n/8$, and the lemma follows in this case. Suppose that $|R(f_n)| < n/8$. We prove that during the first $\delta n/8$ steps of the 1--choice process, at most $\delta |R(f_n)|/2$ elements from $R(f_n)$ have been selected a.a.s. Indeed, for every positive integer $t\le \delta n/8$, the 1-choice process selects an element from $R(f_n)$ with probability at most
\begin{equation*}
\dfrac{|R(f_n)|}{n-t+1}\le \dfrac{|R(f_n)|}{\left(1-\delta/8\right)n} \le \dfrac{2|R(f_n)|}{n}.
\end{equation*}
Since any step is made independently of all previous steps conditionally on the set of already selected bits, the number of elements in $R(f_n)$ selected after the first $\delta n/8$ steps is stochastically dominated by a binomial random variable $\mathrm{Bin}(\delta n/8, 2|R(f_n)|/n)$. Thus, since $|R(f_n)| = \omega(1)$, by Chernoff's bound a.a.s.\ there are at most $\delta |R(f_n)|/2$ elements of $R(f_n)$ selected after the first $\delta n/8$ steps. Hence, since by assumption $T_1(\tilde{f}_n) \ge \delta|R(f_n)|$, $\tilde{f}_n$ (and therefore $f_n$ as well) is activated with probability less than $1/2$ after the first $\delta n/8$ steps, which proves the hypothesis of Theorem \ref{thm:slow}, and the corollary follows.
\end{proof}

%==================
\subsection{Fast sequences}

We present the proofs of Theorem \ref{thm:fast} and Corollary \ref{cor:fast}.

\begin{lemma}\label{lem:tight}
Fix an integer $r\ge 1$ and a sequence of monotone Boolean functions $(f_n)_{n\ge 1}$ satisfying $1\ll T_1(\tilde{f}_n)\ll |R(f_n)|\ll n$. Then, for every $\delta > 0$,
\begin{equation*}
\dfrac{(1-\delta)T_1(\tilde{f}_n)n}{r|R(f_n)|} \le T_r(f_n)\le \dfrac{(1+\delta)T_1(\tilde{f}_n)n}{r|R(f_n)|}.
\end{equation*}
\end{lemma}
\begin{proof}
First, we prove the lower bound. Define
\begin{equation*}
    k^- = k^-(n) = \dfrac{(1-\delta)T_1(\tilde{f}_n) n}{r|R(f_n)|},
\end{equation*}
and let $(Z_i)_{i\ge 1}$ be an infinite sequence of independent Bernoulli random variables with parameter $p_x = \frac{|R(f_n)|}{n-rk^-}$. Let $A_r(t,f_n)$ be the number of activated bits in $R(f_n)$ after $t$ steps of the $r$-complete process and $A_1(t,f_n)$ be the same quantity for the $1$-choice process. We have
\begin{align*}
\Pr{A_r(k^-,f_n) \geq \left(1-\dfrac{\delta}{2}\right)T_1(\tilde{f}_n)} 
&\le \Pr{A_1(rk^-,f_n) \geq \left(1-\dfrac{\delta}{2}\right)T_1(\tilde{f}_n)} \\ 
&\le \Pr{\sum_{i=1}^{rk^{-}} Z_i\ge \left(1-\dfrac{\delta}{2}\right)T_1(\tilde{f}_n)} \\
&= \exp(-\Omega_{\delta}(T_1(\tilde{f_n})) = o(1),
\end{align*}
where the penultimate equality follows from Chernoff's bound and uses that $rk^-p_x = (1-\delta+o(1))T_1(\tilde{f}_n)$, which follows from our assumption that $T_1(\tilde{f}_n) = o(|R(f_n)|)$.

Thus, a.a.s.\ there are at most $(1-\delta/2)T_1(\tilde{f}_n)$ bits in $R(f_n)$ selected during the first $rk^-$ steps by the 1--choice process, so
\begin{equation*}
rk^- \le T_1(f_n) \le rT_r(f_n),
\end{equation*}
which proves the lower bound.

For the upper bound, define 
\begin{equation*}
    k^+ = k^+(n) = \dfrac{(1+\delta)T_1(\tilde{f}_n) n}{r|R(f_n)|}.
\end{equation*}
We prove that, out of the first $k^+$ steps, there are a.a.s.\ at least $(1+\delta/2)T_1(\tilde{f}_n)$ steps such that at least one element in $R(f_n)$ is proposed. Denote by $T$ the hitting time of the above event. Also, recall that $C_1, C_2, \dots, C_{k^+}$ are the sets of size $r$ of elements, proposed during the first $k^+$ steps of the $r$--choice process. Now, let $(Y_i)_{i\ge 1}$ be an infinite sequence of Bernoulli random variables with parameter
\begin{equation*}
p_y = 1 - \left(1 - \frac{|R(f_n)|-(1+\delta)T_1(\tilde{f}_n)}{n}\right)^r.
\end{equation*}
Note that $p_y$ bounds from below the probability that $C_t$ contains an element of $R(f_n)$ for every $t\le T$, and since $|R(f_n)| = o(n)$, $p_y = (1+o(1))\frac{r|R(f_n)|-(1+\delta)T_1(\tilde{f}_n)}{n}$. Thus,
\begin{align*}
\Pr{T\ge k^+} \le\hspace{0.3em} & \Pr{|\{t\le k^+: C_t\cap R(f_n) \neq \emptyset\}| < (1+\delta/2)T_1(\tilde{f}_n)}\\
\le\hspace{0.3em}
& \Pr{\sum_{i=1}^{k^+} Y_i < (1+\delta/2)T_1(\tilde{f}_n)} = \exp(-\Omega_{\delta}(T_1(\tilde{f}_n))) = o(1),
\end{align*}
where the penultimate equality follows from Chernoff's bound and uses that $k^+p_y = (1+\delta+o(1))T_1(\tilde{f}_n)$, which follows from our assumption that $T_1(\tilde{f}_n) = o(|R(f_n)|)$.

We conclude that after $k^+$ steps in the $r$--choice process, at least $(1+\delta/2)T_1(\tilde{f}_n)$ elements of $R(f_n)$ have been selected. Moreover, if at every step $t\le k^+$ we impose on the agent to select an element from $C_t\cap R(f_n)$ uniformly at random if $|C_t\cap R(f_n)|\ge 2$, the set of selected elements in $R(f_n)$ after $k^+$ steps is uniform conditionally on its size. This proves the upper bound.
\end{proof}

\begin{proof}[Proof of Theorem~\ref{thm:fast}]
By Lemma \ref{lem:tight} applied with $r=1$ we have that for every $\delta > 0$, 
\begin{equation*}
    \dfrac{(1-\delta)T_1(\tilde{f}_n)n}{|R(f_n)|} \le T_1(f_n)\le \dfrac{(1+\delta)T_1(\tilde{f}_n)n}{|R(f_n)|}
\end{equation*}
and for every $\delta > 0$ and $r\ge 2$, once again by Lemma~\ref{lem:tight},
\begin{equation*}
    \dfrac{(1-\delta)T_1(\tilde{f}_n)n}{r|R(f_n)|} \le T_r(f_n)\le \dfrac{(1+\delta)T_1(\tilde{f}_n)n}{r|R(f_n)|}.
\end{equation*}
We deduce that for every $\delta > 0$ and $r\ge 2$,
\begin{equation*}
    \dfrac{T_1(f_n)}{r} \le \dfrac{(1+\delta)T_1(\tilde{f}_n)n}{r|R(f_n)|} = \dfrac{1+\delta}{1-\delta} \dfrac{(1-\delta)T_1(\tilde{f}_n)n}{r|R(f_n)|} \le \dfrac{1+\delta}{1-\delta} T_r(f_n)
\end{equation*}
and
\begin{equation*}
    T_r(f_n) \le \dfrac{(1+\delta)T_1(\tilde{f}_n)n}{r|R(f_n)|} = \dfrac{1+\delta}{1-\delta} \dfrac{(1-\delta)T_1(\tilde{f}_n)n}{r|R(f_n)|} \le \dfrac{1+\delta}{1-\delta} \dfrac{T_1(f_n)}{r}.
\end{equation*}
Since the above two chains of inequalities hold for every $\delta > 0$, this proves the theorem.
\end{proof}

\begin{remark}
The hypothesis $|R(f_n)| = o(n)$ in the first point of the theorem cannot be spared. We show this by a counterexample. Fix $\eps\in (0,1]$ and let $J = [\lfloor \eps n\rfloor]$. Let $f_n$ be activated when $\lfloor \log n\rfloor$ of the elements in $J$ are activated, that is, 
\begin{equation*}
f(v) = \left(\sum_{1\le i_1 < \ldots < i_r < \ldots < i_{\lfloor\log n\rfloor}\le \lfloor \eps n\rfloor} \mathds 1_{\forall j\in [\lfloor\log n\rfloor], v_{i_j}=1}\right)\wedge 1.
\end{equation*}
Instead of presenting the (rather direct) computation in this particular case, we choose to explain the logic behind the phenomenon. At any step in the process, there is a positive probability (which is $1 - (1-\eps)^r - r\eps(1-\eps)^{r-1}+o(1)$) that two or more of the randomly proposed $r$ elements are in $J$. Since one may select only one element at a time, one may roughly think that ``one possibility of selecting an element in $R(f_n)$ is missed'' on the above event. Since the number of steps is a.a.s.\ $\Theta(\log n)$, in a constant proportion of all steps (which is $1 - (1-\eps)^r - r\eps(1-\eps)^{r-1}+o(1)$) at least one element of $R(f_n)$ is ``missed'' a.a.s., which causes a delay in the $r$--choice process.
\end{remark}

\begin{remark}
In general, the hypothesis $T_1(f_n)=\omega(1)$ cannot be spared either. If there is a constant $M > 0$ such that for infinitely many $n\in \mathbb N$ one has $T_1(f_n)\le M$, then clearly there cannot be acceleration by a factor of $r+o(1)$ for any $r > M$.
\end{remark}

\begin{proof}[Proof of Corollary~\ref{cor:fast}]
Fix a sequence of positive real numbers $(\eps_k)_{k\ge 1}$ that tends to zero. By assumption, a.a.s., for every $k\ge 1$ there is a sequence of positive integers $(\gamma_{k,n})_{n\ge 1}$ such that, for every $k\ge 1$ and every large enough $n$, with probability at least $1-\varepsilon_k$ the sequence of functions $(f^{\gamma_{k,n}}_n)_{n\ge 1}$ satisfies:
\begin{enumerate}
    \item[(i)] $\gamma_{k,n} \le \eps_k T_1(f^{\gamma_{k,n}}_n) = o_k(T_1(f^{\gamma_{k,n}}_n))$,
    \item[(ii)] $|R(f^{\gamma_{k,n}}_n)| \le \eps_k n = o_k(n)$,
    \item[(iii)] $T_1(\tilde{f}^{\gamma_{k,n}}_n) \le \eps_k |R(f^{\gamma_{k,n}}_n)| = o_k(|R(f^{\gamma_{k,n}}_n)|)$ and $T_1(\tilde{f}^{\gamma_{k,n}}_n) = \omega_k(1)$.
\end{enumerate}
Thus, a.a.s.\ one may find a sequence $(k(n))_{n\ge 1}$ satisfying $k(n) = \omega(1)$ such that, for every large enough $n\ge 1$, the sequence $(\gamma_n)_{n\ge 1} = (\gamma_{k(n),n})_{n\ge 1}$ satisfies
\begin{enumerate}
    \item $\gamma_n \le \eps_{k(n)} T_1(f^{\gamma_n}_n) = o(T_1(f^{\gamma_n}_n))$,
    \item $|R(f^{\gamma_n}_n)| \le \eps_{k(n)} n = o(n)$,
    \item $T_1(\tilde{f}^{\gamma_n}_n) \le \eps_{k(n)} |R(f^{\gamma_n}_n)| = o(|R(f^{\gamma_n}_n)|)$ and $T_1(\tilde{f}^{\gamma_n}_n) = \omega(1)$.
\end{enumerate}

By using conditions (2) and (3), a direct application of Theorem \ref{thm:fast} for the sequence of Boolean functions $(f^{\gamma_n}_n)_{n\ge 1}$ shows that $T_r(f^{\gamma_n}_n) = (1+o(1))T_1(f^{\gamma_n}_n)/r$ a.a.s. On the other hand, $\mathbb E[T_1(f_n) - T_1(f_n^{\gamma_n})] = \gamma_n$ (note that since the $\gamma_n$ bits are chosen uniformly at random, the expected number of additional rounds needed to obtain probability at least $1/2$ for $f_n$ to evaluate to 1 is $T_1(f_n)-\gamma_n$), so since $\gamma_n = o(T_1(f_n^{\gamma_n}))$, one may conclude by Markov's inequality for $T_1(f_n) - T_1(f_n^{\gamma_n})$ that $T_1(f_n) = (1+o(1))T_1(f_n^{\gamma_n})$ a.a.s., and similarly $T_r(f_n) = (1+o(1)) T_r(f_n^{\gamma_n})$ a.a.s., which concludes the proof of the corollary.
\end{proof}

%=======================================
\section{Applications}\label{sec:app}
In this section we give several examples of application of Theorems~\ref{thm:slow}~and~\ref{thm:fast} and the corresponding corollaries.

%==================
\subsection{Juntas}

A Boolean function $f$ is an \emph{$M$-junta} if $|R(f)| \leq M$. Fix a sequence $(f_n)_{n\ge 1}$ of monotone Boolean functions such that there is $M\in \N$ satisfying
\begin{equation*}
\max_{n\in \N} |R(f_n)|\le M. 
\end{equation*}
Fix also a positive constant $c = c(M)$ satisfying $M\log((1-c)^{-1})\le 1/2$. Under the above assumption for every sufficiently large $n\ge 1$ we deduce that $T_1(f_n)\ge c n$: indeed, the probability not to encounter any element of $R(f_n)$ during the first $c n$ steps of the 1-choice process is bounded from below by
\begin{align*}
    \prod_{i=0}^{c n-1} \left(1-\dfrac{|R(f_n)|}{n-i}\right)
    &\ge \prod_{i=0}^{c n-1} \left(1-\dfrac{M}{n-i}\right) \\
    &= \exp\left(-M\log\left(\frac{1}{1-c}\right) + o(1)\right) \\
    &\ge \exp\left(-\dfrac{1}{2}+o(1)\right),
\end{align*}
which is larger than $1/2$ for every large enough $n$. We conclude by Theorem~\ref{thm:slow} that there is $C = C(r, (f_n)_{n\ge 1}) > 0$ such that, for every sufficiently large $n$,
\begin{equation*}
  T_r(f_n) \ge C n + \frac{T_1(f_n)}{r},
\end{equation*}
and hence sequences of $M$--juntas are slow for any $M\in \mathbb N$.

\subsection{Recursive Majority}

Consider two positive integer sequences $(k_n)_{n\ge 1}$ and $(t_n)_{n\ge 1}$ such that, for every $n\ge 1$, $k_n$ is odd and $k_n\ge 3$, and $(k_n^{t_n})_{n\ge 1}$ is an increasing sequence that tends to infinity as $n\to +\infty$. Fix $n\in \mathbb N$ and denote $k = k_n, t = t_n$ and $N = k^t$. Now, define the sets $(S_i^j)_{j\in \{0,\dots,t\}, i\in k^{t-j}}$ where, for every $j\in \{0,\dots,t\}$ and $i\in k^{t-j}$, $S_i^j = \{ik^j-(k^j-1),\dots, ik^j\}$. Note that, for every $j\in [t]$ and $i\in [k^{t-j}]$,
$S_{ki-(k-1)}^{j-1} \cupdot \ldots \cupdot S_{ki}^{j-1}$.

Now, for $i\in [k^t]$, we say that the set $S_i^0$ is \emph{activated} if the bit $i$ is turned from 0 to 1, and for every $j\in [t]$ and $i\in [k^{t-j}]$, $S_i^j$ is \emph{activated} if at least $\frac{k+1}{2}$ of the sets $S_{ki-(k-1)}^{j-1}, \dots, S_{ki}^{j-1}$ are activated. Define $f_N:\bm{x}\in \{0,1\}^N\mapsto \mathds{1}_{S_1^t \text{ is activated by }\bm{x}}$, that is, if only the 1-bits of $\bm{x}$ have been activated, then $S_1^t$ is activated as well.

The following lemma shows that, for any $s < N/2$, evaluating $f_N$ at a uniformly chosen vector conditioned to have exactly $s$ 1-bits yields 1 with probability strictly smaller than $1/2$.

\begin{lemma}\label{lem rec maj}
Fix a random variable $X$ distributed uniformly over $\{0,1\}^N$. Then, $\Pr{f_N(X)=1}\le 1/2$. Moreover, $\Pr{f_N(X)=1\mid ||X||_1 = s} < 1/2$ for $s < N/2$. 
\end{lemma}
\begin{proof}
Note that if $\bm{x}\in \{0,1\}^N$ satisfies $f_N(\bm{x})=1$, then $f_N(\bm{1}-\bm{x})=0$, which shows the first statement. For the second statement, we will need the following theorem, which is a special case of a more general result that one may trace back to Sperner~\cite{Spe}, see also~\cite{SW2020}.
\begin{theorem}[Special case of the local LYM inequality]\label{LYM ineq}
Fix $A\subseteq \binom{[N]}{\lfloor N/2\rfloor}$ and denote
\begin{equation*}
\partial A = \left\{S\in \binom{[N]}{\lceil N/2\rceil}\hspace{0.3em}\Bigg|\hspace{0.3em} \exists S'\in A, S'\subseteq S\right\}.
\end{equation*}
Then, $|A|\le |\partial A|$. Moreover, equality holds if and only if $A = \emptyset$ or $A = \binom{[N]}{\lfloor N/2\rfloor}$.
\end{theorem}

Denote by $A$ the set of vectors $\bm{x}\in \{0,1\}^N$ containing $\lfloor N/2\rfloor$ 1-bits and satisfying $f_N(\bm{x}) = 1$. Since $A$ is neither empty nor contains all vectors with exactly $\lfloor N/2\rfloor$ 1-bits, by Theorem~\ref{LYM ineq} $|A| < |\partial A|$. On the other hand, by the same symmetry considerations as above, the number of vectors $\bm{x}\in \{0,1\}^N$ with either $\lfloor N/2\rfloor$ or $\lceil N/2\rceil$ 1-bits such that $f_N(\bm{x})=1$ is $\binom{N}{\lfloor N/2\rfloor}$. We conclude that
\begin{align*}
    \Pr{f_N(X)=1\mid ||X||_1 = s} 
    &\le\hspace{0.3em} \Pr{f_N(X)=1\mid ||X||_1 = \lfloor N/2\rfloor}\\
    &=\hspace{0.3em} |A|\binom{N}{\lfloor N/2\rfloor}^{-1}\\
    &<\hspace{0.3em} \dfrac{|A|+|\partial A|}{2}\binom{N}{\lfloor N/2\rfloor}^{-1}\\
    &\le\hspace{0.3em} \Pr{f_N(X)=1\mid \lfloor N/2\rfloor\le ||X||_1 \ge \lceil N/2\rceil} = \frac{1}{2},
\end{align*}
which finishes the proof of the second statement.
\end{proof}

By Lemma~\ref{lem rec maj} we conclude that for every $n\ge 1$ one has $T_1(f_N)\ge N/2$, so by Theorem~\ref{thm:slow} we deduce that there exists $C = C((f_N))$ such that, for every $n\ge 1$, $f_N = f_{N(n)}$ satisfies
\begin{equation*}
    T_r(f_N)\ge C N + \frac{T_1(f_N)}{r},
\end{equation*}
and hence recursive majorities is slow.

%==================
\subsection{Tribes}
Let $(s_n)_{n\ge 1}$ be a sequence of positive integers such that, for every $n\in \mathbb N$, $s_n\in [1,n]$. For every $n\in \mathbb N$, write $n = s_nt_n + r_n$, where $r_n\in \{0, \dots, s_n-1\}$ is the remainder of the division of $n$ by $s_n$. Then, for every $n\in \mathbb N$, given $s_n$, a \emph{tribe partition} of $[n]$ is a $t_n$-tuple of sets $(S_1, S_2\dots, S_{t_n})$ such that $S_1\cupdot S_2\cupdot \dots \cupdot S_{t_n} = n$ and for every $i\in [t_n]$, $|S_i|\in \{s_n, s_n+1\}$. For every $n\in \mathbb N$, a \emph{tribe function of tribe size $s_n$} associated to the tribe partition $(S_1, S_2\dots, S_{t_n})$ is a function
\begin{equation*}
    f_n: \bm{x}\in \{0,1\}^n\mapsto \mathds{1}_{\exists 1 \le i \le t_n, \text{all bits in positions }S_i \text{ in } \bm{x} \text{ are 1} }.
\end{equation*}

\begin{lemma}
Fix any $\delta > 0$ and a sequence of tribe functions $(f_n)_{n \ge 1}$ of tribe sizes $(s_n)_{n\ge 1}$ satisfying that for all $n\in \mathbb N, s_n \ge \delta\log n$. Then, there is a constant $C = C(\delta, (f_n)) > 0$ such that $T_r(f_n)\ge Cn + T_1(f_n)/r$.
\end{lemma}
\begin{proof}
Fix $p = \exp(-1/\delta)$. We first show that a.a.s.\ $f_n$ is not activated if every bit is put to 1 with probability $p$ independently of all other bits. Indeed, the probability of the above event is bounded from below by
\begin{equation*}
    (1 - p^{s_n})^{t_n} = \exp(-(1+o(1)) p^{s_n} n/s_n) \ge \exp(-(1+o(1))/s_n) = 1+o(1).
\end{equation*}

Moreover, by Chernoff's inequality a.a.s. at least $pn/2$ bits are put to 1. We conclude that $T_1(f_n)\ge pn/2$ for every large enough $n$, which allows us to conclude by Theorem~\ref{thm:slow}.
\end{proof}

%==================
\subsection{Connectivity and \texorpdfstring{$k$}{k}--connectivity}

For every $n\ge 1$, consider an ordering $\mathcal{I}_n$ of the set of pairs of vertices of $K_n$. Let $g_n$ be a function from $\{0,1\}^{\binom{n}{2}}$ to the set of graphs on $n$ vertices such that, for every $\bm v\in \{0,1\}^{\binom{n}{2}}$, the $i$--th pair of vertices of $\mathcal{I}_n$ is an edge in $g_n(\bm v)$ if $\bm v_i = 1$, and is not an edge if $\bm v_i = 0$. Define 
$$f_n:\bm v\in \{0,1\}^{\binom{n}{2}} \mapsto \mathds{1}_{g_n(\bm v) \text{ is connected}}.$$

Clearly, for every $n\ge 1$, the function $f_n$ is monotone and all $\binom{n}{2}$ vertex pairs of $K_n$ belong to $R(f_n)$ (note that any set of $\binom{n}{2}-1$ edges does not decide if a graph is connected or not in general). It is well known that for the binomial random graph $G(n,p)$ connectivity undergoes a sharp threshold at $p=(1+o(1))\log n/n$, coinciding with the moment when the last isolated vertex becomes incident to an edge. We now show that the sequence $(f_n)_{n\ge 1}$ fixed above is accelerated by a factor of $r+o(1)$ in the $r$--choice process (note that this also holds for the threshold of disappearance of the last isolated vertex):

\begin{lemma}\label{lem connectivity}
The sequence $(f_n)_{n\ge 1}$ defined above is fast.
\end{lemma}
\begin{proof}
Fix any $r\in \mathbb N$. Consider the following strategy for the $r$--choice process: at each of the first $s=n \log \log n$ steps, select an arbitrary edge among the $r$ proposed ones. Then, at any step, select an edge that contains at most one vertex in the largest connected component if possible, and select an arbitrary edge otherwise. It is well known (see e.g.~\cite{Bol}) that, after $s$ steps, a.a.s.\ the graph consists of a giant component that contains all but $o(n)\cap \omega(n^{1/2})$ of all $n$ vertices as well as $\omega(n^{1/2})$ isolated vertices. Hence, after $s$ steps, a.a.s.\ only $o(n^2)\cap \omega(n^{3/2})$ of the remaining 0-bits may change the connectivity (namely the bits corresponding to the edges incident to at least one vertex outside the giant component). We condition on this event. Suppose that the first $s$ activated bits have indices $i_1 < i_2 < \dots <i_s$ (which is a uniform random set of $s$ out of all $\binom{n}{2}$ bits). Denote by $f_n^s$ the (random) restriction of $f_n$ over the set of vectors in $\{0,1\}^{\binom{n}{2}}$ such that each of the bits with indices $i_1 < i_2 < \dots <i_s$ is turned to 1. Hence $|R(f^s_n)|=o(n^2)\cap \omega(n^{3/2})$. Since $T_1(\tilde{f}^s_n) \le T_1(f_n) = \Theta(n \log n)$ (for the sharp threshold for connectivity ensuring the last equality, see again~\cite{Bol}) and $T_1(\tilde{f}^s_n)\ge n^{1/2}/2 = \omega(1)$ by our conditioning, we have $1\ll T_1(\tilde{f}_n^s)\ll |R(f_n^s)|\ll n^2$, so by Theorem~\ref{thm:fast} $T_r(f_n^s) = (1+o(1))T_1(f_n^s)/r$. Moreover, before the conditioning we have $\mathbb E[T_1(f_n) - T_1(f_n^s)] = s = o(n\log n) = o(T_1(f_n))$, and the same holds for $T_r(f_n) - T_r(f_n^s)$. By Markov's inequality we conclude that both $T_1(f_n) = (1+o(1))T_1(f_n^s)$ and $T_r(f_n) = (1+o(1))T_r(f_n^s)$ a.a.s., so $T_r(f_n) = (1+o(1))T_1(f_n)/r$, which proves the lemma.
\end{proof}

\begin{remark}
For any $k \ge 2$, a graph is said to be $k$--connected if the deletion of any $k-1$ vertices leaves a connected graph. Also, the $k$--core of a graph $G$ is the largest subgraph of $G$ with minimum degree $k$. It is well-known that a sharp threshold for $k$--connectivity occurs at $p = (\log n + (k-1)\log\log n)/n$ (see Theorem 7.7 of~\cite{Bol}) as well as the fact that after $n\log\log n$ steps of the 1-choice process the $k$--core of the resulting random graph contains $n+o(n)$ vertices and is $k$--connected a.a.s.\ (see again~\cite{Bol}). Hence, a straightforward modification of the proof of Lemma~\ref{lem connectivity} shows that, for every $r\ge 2$, $k$--connectivity is accelerated by a factor of $r(1+o(1))$ by the $r$--choice process.
\end{remark}

\begin{remark}
The appearance of both Perfect matching and Hamilton cycle on $n$ vertices fall into the category of monotone functions $f_n$ satisfying $|R(f_n)|=\binom{n}{2}$ for which the $r$--choice process therefore gives a $r(1+o(1))$--factor acceleration, see~\cite{Krive} for Hamilton cycle (and as they remark in Section 5, Point 4, their result also applies to Perfect matching). Unfortunately our results do not lead to a significant simplification of their argument.
\end{remark}

%=======================================
%\section*{Acknowledgments}
%We would like to acknowledge the assistance of volunteers in putting together this example manuscript.

%=======================================
\bibliographystyle{siamplain}
\bibliography{References}
\end{document}